\documentclass[reqno]{amsart}
\usepackage{amssymb, amsmath, amsthm, amscd, mathrsfs}

\usepackage{paralist}
\usepackage{cite}
\usepackage{bigints}
\usepackage{dsfont}
\usepackage{empheq}
\usepackage{amssymb}
\usepackage{cases}
\usepackage{enumitem}
\allowdisplaybreaks

\usepackage{caption}
\usepackage{booktabs}
\usepackage{xcolor}

\usepackage{float}
\usepackage[ruled,vlined,linesnumbered,noresetcount]{algorithm2e}

\usepackage{hyperref}

\usepackage[top=2.8cm, bottom=3cm, left=3cm, right=3cm]{geometry}

\theoremstyle{plain}
\newtheorem{theorem}{Theorem}[section]

\newtheorem{lemma}[theorem]{Lemma}
\newtheorem{proposition}[theorem]{Proposition}
\theoremstyle{definition}
\newtheorem{definition}[theorem]{Definition}
\newtheorem{remark}[theorem]{Remark}
\numberwithin{equation}{section}

\title[Boundary null controllability of the heat equation] 
      {Boundary null controllability of the heat equation with Wentzell boundary condition and Dirichlet control}

\author{S.E. Chorfi}
\author{M.I. Ismailov}
\author{L. Maniar}
\author{I. Öner}

\address{S.E. Chorfi, Faculty of Sciences Semlalia, Cadi Ayyad University, LMDP, UMMISCO (IRD-UPMC), Marrakesh, Morocco}
\email{s.chorfi@uca.ac.ma}
\address{L. Maniar, Faculty of Sciences Semlalia, Cadi Ayyad University, LMDP, UMMISCO (IRD-UPMC), Marrakesh and University Mohammed VI Polytechnic, Vanguard Center, Benguerir, Morocco}

\email{maniar@uca.ac.ma}

\address{M.I. Ismailov, I. Oner, Department of Mathematics, Gebze Technical University, 41400 Gebze-Kocaeli, Turkey and Center for Mathematics and its Apllications, Khazar University, Azerbaijan}
\email{mismailov@gtu.edu.tr}
\email{ioner@gtu.edu.tr}

\makeatletter 
\@namedef{subjclassname@2020}{%
  \textup{2020} Mathematics Subject Classification}
\makeatother

 
\begin{document}
\begin{abstract}
We consider the linear heat equation with a Wentzell-type boundary condition and a Dirichlet control. Such a boundary condition can be reformulated as one of dynamic type. First, we formulate the boundary controllability problem of the system within the framework of boundary control systems, proving its well-posedness. Then we reduce the question to a moment problem. Using the spectral analysis of the associated Sturm-Liouville problem and the moment method, we establish the null controllability of the system at any positive time $T$. Finally, we approximate minimum energy controls by a penalized HUM approach. This allows us to validate the theoretical controllability results obtained by the moment method.
\end{abstract}

\dedicatory{\large Dedicated to the memory of Professor Hammadi Bouslous}

\maketitle

\section{Introduction}
Let $T>0$ be the time horizon. In this article, we study the boundary null controllability problem of heat equation with a Wentzell boundary condition: 
\begin{equation}\label{eq:1}
\begin{cases}
u_{t}(x,t)=u_{xx}(x,t),\quad \text{ in }  D_T, \\ 
u(0,t)=f(t),\quad 0<t<T, \\ 
au_{xx}(1,t)+du_{x}(1,t)-bu(1,t)=0,\quad 0<t< T, \\ 
u(x,0)=u_0(x),\,  \text{ in }  \Omega, \; u(1,0)=u_{0,1}, 
\end{cases}
\end{equation}
where $D_T=\Omega\times (0,T)$, $\Omega=(0, 1),$ $u_0(x)\in L^2(\Omega),$ $u_{0,1} \in \mathbb R,$ are the initial states, $f(t)\in L^2(0,T)$ is the control function, and $a,b,d$ are given numbers such that $ad>0$. Our approach relies on the moment method where the main difficulty comes from the Wentzell boundary condition $\eqref{eq:1}_3$ that includes the eigenvalue parameter. Note that by using $\eqref{eq:1}_1$ for a regular solution, the Wentzell condition $\eqref{eq:1}_3$ takes the form
$$au_{t}(1,t)+du_{x}(1,t)-bu(1,t)=0,$$
which is a dynamic boundary condition involving the time derivative on the boundary.

Boundary conditions of Wentzell type originate independently from the works of Wentzell \cite{We56} and Feller \cite{Fe57} in the context of generalized differential operators of second order and the link with Markov processes. A more general setting was considered in \cite{We59}. There has been significant interest in evolution equations with Wentzell and dynamic boundary conditions in recent years. These equations have found applications in various mathematical models, including population dynamics, heat transfer and continuum thermodynamics. They are particularly relevant when considering diffusion processes occurring at the interface between a solid and a moving fluid. References \cite{Farkas:2011, Langer:1932, Elst:2013, Go19} and the cited literature provide further insights into these applications. For a comprehensive understanding of the physical interpretation and derivation of such boundary conditions, the reader could refer to the seminal paper \cite{Goldstein:2006}.
 
Although the null controllability of parabolic equations with static boundary conditions (Robin, Neumann, Dirichlet, periodic, semi-periodic) has been extensively studied, see, e.g. \cite{Emanuilov:95, Guo:95, Martin:2016, Rousseau:2007, Cara:2006, AmmarKhodja:2017, Ismailov:2023a, Oner:2023}, there has been a recent focus on the null controllability of parabolic systems with dynamic boundary conditions. This topic has been investigated in recent studies, including \cite{Kumpf:2004, Maniar:2017, Khoutaibi:2020, Chorfi:2023}. Kumpf and Nickel \cite{Kumpf:2004} presented abstract results in which they applied semigroup theory to obtain approximate boundary controllability of the one-dimensional heat equation with dynamic boundary conditions and boundary control. Later, Maniar et al. \cite{Maniar:2017} established the null controllability of multi-dimensional linear and semilinear heat equations with dynamic boundary conditions of surface diffusion type. Their approach involves introducing a new Carleman estimate specifically designed to address these particular boundary conditions. In \cite{Khoutaibi:2020}, the authors considered the heat equation in a bounded domain with a distributed control supported on a small open subset, subject to dynamic boundary conditions of surface reaction-diffusion type involving drift terms in bulk and on the boundary. They proved that the system is null controllable at any time, relying on new Carleman estimates designed for this type of boundary conditions. Furthermore, Khoutaibi et al. \cite{Khoutaibi:2022} investigated the null controllability of the semilinear heat equation with dynamic boundary conditions of the surface reaction-diffusion type, with nonlinearities involving drift terms. Chorfi et al. \cite{Chorfi:2023} focused on the boundary controllability of heat equation with dynamic boundary conditions. They proved that the equation is null controllable at any time using a boundary control supported on an arbitrary sub-boundary. The main result was established through a new boundary Carleman estimate and regularity estimates for the adjoint system. Additionally, Mercado and Morales have recently studied the exact controllability of a Schrödinger equation with dynamic boundary conditions in \cite{Mercado:2023}.

In contrast to the studies above, in this article, we focus on the boundary null controllability of the heat equation with a Wentzell boundary condition and a Dirichlet control by applying the moment method \cite{Fattorini:71}, \cite{Fattorini:74}. Although this method proves to be highly effective in solving such problems for one-dimensional systems, it necessitates a comprehensive understanding and proficiency in the spectral properties of the system. Moreover, we propose a constructive algorithm to compute approximate controls of minimum energy by adapting the penalized Hilbert Uniqueness Method (HUM) with a Conjugate Gradient (CG) scheme.

The paper is structured as follows. The well-posedness of the boundary control system is presented in Section \ref{sec2}. The boundary null controllability problem is reduced to an equivalent problem in Section \ref{sec3}. Then we present the null controllability proof by the moment method. 
In Section \ref{sec5}, we propose an algorithm to approximate the controls and perform some numerical simulations to endorse our theoretical results. Finally, in Section \ref{sec6}, we give a conclusion and some perspectives.

\section{Well-posedness}\label{sec2}
We shall study the well-posedness of the boundary control problem \eqref{eq:1}. Here, we borrow our terminology from \cite{TW09} and \cite{Ou'05}. We consider the Hilbert space $H=L^2(\Omega,\mathrm{d}x) \oplus \mathbb{R}$ with the inner product defined as follows:
$$
((u,u_1), (v,v_1))_H=\int_0^1 u v d x+\frac{a}{d} u_1 v_1 .
$$
One can identify the space $H$ with the space $L^2(\Omega_1,\mathrm{d}\nu)$ isomorphically, where $\Omega_1:=(0,1]$ and $\mathrm{d}\nu:=\mathrm{d}x \oplus \dfrac{a}{d} \mathrm{d}\delta_1$, $\delta_1$ is the Dirac measure at $x=1$.

System \eqref{eq:1} can be rewritten as the following boundary control system
\begin{equation}\label{cp}
\left\{\begin{array}{l}
U^{\prime}(t)=\mathcal{L} U(t)\\
\mathcal{G}U(t)=f(t), \qquad t \in(0, T), \\
U(0)=U_0,
\end{array}\right.
\end{equation}
where
$$
\begin{gathered}
\mathcal{L}(u,u_1):=\left(\begin{array}{c}
u_{xx} \\
-\frac{d}{a}u_x(1) + \frac{b}{a} u_1
\end{array}\right), \; D\left(\mathcal{L}\right):=\mathcal{U}=\left\{(u, u_1)\in H^1(0, 1)\oplus \mathbb{R}: u(1)=u_1\right\}, \\
\mathcal{G}(u,u_1)=u(0), \; D(\mathcal{G})=\mathcal{U}, \qquad U(t):=\left(\begin{array}{c}
u(\cdot,t) \\
u(1,t)
\end{array}\right), \quad U_0:=\left(\begin{array}{c}
u_0\\
u_{0,1}
\end{array}\right).
\end{gathered}
$$
Let $\mathcal{U}'$ denote the dual space of $\mathcal{U}$ with respect to the pivot space $H$. The operator $\mathcal{L}\colon \mathcal{U} \longrightarrow  \mathcal{U}'$ (defined in the weak sense) is continuous, and the operator $\mathcal{G}\colon \mathcal{U} \to \mathbb R$ is onto. Considering the spaces
$$H_s=\left\{(u, u_1)\in H^s(0, 1)\oplus \mathbb{R}: u(1)=u_1, \; u(0)=0\right\}, \quad s>\frac{1}{2},$$
we have  $\ker \mathcal{G}=H_1$.

To prove the well-posedness of the boundary control problem \eqref{cp}, we consider the state space $\mathcal{H}=H_{-1}$, where $H_{-k}$, $k=1,2$, denotes the dual space of $H_k$ with respect to the pivot space $H$. Let $\mathcal{A}_{-1}=\mathcal{L}\rvert_{H_1}$ be the part of $\mathcal{L}$ in $H_1$ and let $\mathcal{A}=\mathcal{L}\rvert_{H_2}$ be the part of $\mathcal{L}$ in $H_2$.
\begin{proposition}\label{prw}
The operator $\mathcal{A}$ is densely defined, self-adjoint, and generates a compact analytic $C_0$-semigroup $\left(e^{t \mathcal{A}}\right)_{t \geq 0}$ of angle $\dfrac{\pi}{2}$ on $H$.    
\end{proposition}

\begin{proof}
We have $\left\{(u, u_1)\in C^{\infty}([0, 1])\oplus \mathbb R \colon u(1)=u_1, \; u(0)=0\right\} \subset H_2$ is dense in $H$, therefore $\mathcal{A}$ densely defined. To prove the stated properties of $\mathcal{A}$, we introduce the densely bilinear form
$$
\mathfrak{a}[(v, v_1),(w, w_1)]=\int_0^1 v_x w_x d x - \frac{b}{d} v_1 w_1,
$$
with domain $D(\mathfrak{a})=H_1$ in $H$. Given a real number $\mu$, we define bilinear form $\mathfrak{a}+\mu$ by
$$
(\mathfrak{a}+\mu)[(v, v_1),(w, w_1)]=\mathfrak{a}[(v, v_1),(w, w_1)]+\mu\left((v, v_1),(w, w_1)\right)_H.
$$
Using integration by parts and Cauchy-Schwarz inequality, one can choose $\mu \in \mathbb{R}$ such that
$$
(\mathfrak{a}+\mu)[(v, v_1),(v, v_1)] \geq \frac{1}{2}\left\|(v, v_1)\right\|_{H_1}^2 \quad \text { for all } (v, v_1) \in H_1 .
$$
Following \cite{Ou'05}, we can easily check that the form $\mathfrak{a}+\mu$ is densely defined, accretive, symmetric, continuous and closed. Then, we associate to the form $\mathfrak{a}$ the operator $\widetilde{\mathcal{A}}$ defined by
$$
\begin{gathered}
D(\widetilde{\mathcal{A}}):=\left\{\left(v, v_1\right) \in H_1, \text { there exists }\left(z, z_1\right) \in H\right. \text { such that } \\
\left.\mathfrak{a}\left[\left(v, v_1\right),\left(w, w_1\right)\right]=\left(\left(z, z_1\right),\left(w, w_1\right)\right)_{H} \text { for all }\left(z, z_{1}\right) \in H_1\right\}, \\
\widetilde{\mathcal{A}}\left(v, v_1\right):=-\left(z, z_1\right) \quad \text{ for all }\left(v, v_1\right) \in D(\widetilde{\mathcal{A}}).
\end{gathered}
$$
Therefore, the operator $\widetilde{\mathcal{A}}$ is self-adjoint, bounded above, and has compact resolvent (since $H_1 \hookrightarrow H$ is compact). Then it generates a compact analytic $C_0$-semigroup on $H$ of angle $\dfrac{\pi}{2}$ (see \cite[p.~106]{EN00}). By using integration by parts, we can show that $\mathcal{A}=\widetilde{\mathcal{A}}$. This completes the proof.
\end{proof}
Following \cite{AT12}, we also have:
\begin{proposition}
The operator $\mathcal{A}_{-1}$ is densely defined, self-adjoint, and generates a compact analytic $C_0$-semigroup $\left(e^{t \mathcal{A}_{-1}}\right)_{t \geq 0}$ of angle $\dfrac{\pi}{2}$ in $\mathcal{H}$.    
\end{proposition}
Consider the extension $\mathcal{A}_{-1}: H \to H_{-2}$, $\alpha\in \rho(\mathcal{A}_{-1})$, and the control operator defined by
$$\mathcal{B}\kappa:=(\alpha-\mathcal{A}_{-1})\mathcal{D}\kappa, \qquad \kappa\in \mathbb R,$$
where $\mathcal{D}: \mathbb R \longrightarrow  H$ is the Dirichlet map defined by
$$\left(\begin{array}{c}
u \\
u_1
\end{array}\right)=\mathcal{D}\kappa \Longleftrightarrow \begin{cases}u_{xx}=\alpha u & \text { in } \Omega, \\ (b-\alpha a)u_1 =d u_x(1), &  \\ u(0)=\kappa. & \end{cases}$$
System \eqref{cp} can be rewritten
equivalently in the classical form
\begin{equation}\label{cpn}
\left\{\begin{array}{l}
U'(t)=\mathcal{A}_{-1}U(t)+\mathcal{B}f(t), \qquad U'(t)\in \mathcal{H},\\
U(0)=U_0 \in \mathcal{H}.
\end{array}\right.
\end{equation}
 Let $\Psi\in \mathcal{H}_{-1}$ and $\Phi\in  H_1$ and denote $(\Psi, \Phi)_{\mathcal{H}_{-1}, H_1}=((\alpha-\mathcal{A}_{-1})^{-1}\Psi, \Phi)_H$ the duality bracket. Next, we identify $\mathcal{B}^*$, the adjoint  of operator $\mathcal{B}$.
\begin{lemma}\label{ladj}
The operator $\mathcal{B}^*$ is given by
$$\mathcal{B}^*(\psi,\psi_1)=\phi_x(0),  \qquad  (\psi,\psi_1)\in H_1,$$
where $(\phi,\phi_1) \in H_1$ is the solution of
\begin{equation}
    \begin{cases}\phi_{xx}=\alpha \phi - \psi & \text { in } \Omega, \\ a \phi_{xx}(1)-b \phi_1+d \phi_x(1)=0, &  \\ \phi(0)=0. & \end{cases}
\end{equation}
\end{lemma}

\begin{proof}
Let $U$ be a smooth solution to \eqref{cpn} corresponding to a control $\kappa$. Then, for all $\Psi \in H_1$,
\begin{equation}\label{1eq}
    (\alpha U-U', \Psi)_{\mathcal{H}_{-1}, H_1}=((\alpha-\mathcal{A}_{-1})U, \Psi)_{\mathcal{H}_{-1}, H_1} -(\mathcal{B}\kappa, \Psi)_{\mathcal{H}_{-1}, H_1}.
\end{equation}
By \eqref{eq:1} with $f(t)=\kappa$, integration by parts yields, for all $\Phi:=(\phi,\phi_1) \in H_2$,
$$
(\alpha U-U',\Phi)_H=(U, (\alpha-\mathcal{A}_{-1})\Phi)_H - \kappa \phi_x(0).
$$
Choosing $\Phi=(\alpha-\mathcal{A}_{-1})^{-1}\Psi$, \eqref{1eq} implies that
$$
(\mathcal{B}\kappa, \Psi)_{\mathcal{H}_{-1}, H_1}= \kappa \phi_x(0)=\kappa \mathcal{B}^*\Psi,  \qquad \forall \kappa\in \mathbb R,\; \forall \Psi\in H_1.
$$
This completes the proof.
\end{proof}

This yields the following result.
\begin{lemma}
    The boundary control system \eqref{cp} is well-posed in $\mathcal{H}$.
\end{lemma}

\begin{proof}
    Since $\mathcal{A}_{-1}$ generates a symmetric $C_0$-semigroup $\left(e^{t\mathcal{A}_{-1}}\right)_{t\ge 0}$ on $\mathcal{H}$, it suffices to check that the operator $\mathcal{B}$ is admissible for $\left(e^{t\mathcal{A}_{-1}}\right)_{t\ge 0}$. Owing to the fact that $\mathcal{B}: \mathbb R \to H_{-2}$ is continuous, by \cite[Lemma 8]{Tr24}, we conclude that the control operator $\mathcal{B}$ is admissible for $\left(e^{t\mathcal{A}_{-1}}\right)_{t\ge 0}$.
\end{proof}

\begin{remark}
    The boundary control system \eqref{cp} is ill-posed in the space $H$, since for $U_0\in H$ the solution may not belong to $C([0, T]; H)$.
\end{remark}

Consequently, we obtain the well-posedness result of System \eqref{cp}, see \cite[Proposition 10.1.8]{TW09} for the proof.
\begin{proposition}
For every $T>0,\; U_0:=(u_0,u_{0,1}) \in \mathcal{U}$ and $f \in H^1(0, T)$ that satisfy the compatibility condition $u_0(0)=f(0)$, the boundary control problem \eqref{cp} admits a unique solution such that $U \in C([0, T]; \mathcal{U}) \cap C^1([0, T];\mathcal{H}).$
\end{proposition}

Next, we define a (weak) transposition solution to \eqref{eq:1}. Let us introduce the backward system
\begin{equation}\label{tr}
\begin{cases}
\varphi_{t}+\varphi_{xx}=g, \quad \text{in } D_T,\\ 
a\varphi _{xx}(1,t)-b\varphi_1(t)+d\varphi _{x}(1,t)=g_1, \\ 
\varphi(0,t)=0, \\ 
\varphi(\cdot,T)=0, \quad \text{ in } \Omega, \; \varphi(1,T)=0, 
\end{cases}
\end{equation}
where $G:=(g,g_1)\in L^2(0,T;H)$. From Proposition \ref{prw}, System \eqref{tr} admits a unique strong solution satisfying  $\Phi:=(\varphi,\varphi_1) \in L^2(0,T;H_2) \cap C([0,T];H_1)$.
\begin{definition}
    Given $U_0\in \mathcal{H}$ and a control $f\in L^2(0,T)$, a function $U(t):=(u(\cdot,t),u_1(t))\in L^2(0,T;H)$ is said to be a solution by transposition to \eqref{eq:1} if the following identity
    \begin{equation} \label{eqtr}
\int_0^T (U(t),G(t))_{H}dt=(U(0),\Phi(0))_{\mathcal{H},H_1}+\int_{0}^{T}\varphi
_{x}(0,t)f(t)dt
\end{equation} 
holds for any $G\in L^2(0,T;H)$, where $\Phi$ is the corresponding solution to \eqref{tr} and $(\cdot,\cdot)_{\mathcal{H},H_1}$ is the duality bracket.
\end{definition}
This definition is motivated by the calculation for a smooth solution $U$ (see Lemma \ref{lmm:1}). Therefore, following \cite[Appendix A]{Ca10}, one can prove
\begin{proposition}
    For every $U_0 \in \mathcal{H}$ and $f\in L^2(0,T)$, System \eqref{eq:1} admits a unique transposition solution $U\in L^2(0,T;H) \cap C([0,T];\mathcal{H})$ such that
    $$\|U\|_{L^2(0,T;H)}+\|U\|_{C([0,T];\mathcal{H})} \le C\left(\|U_0\|_{\mathcal{H}}+\|f\|_{L^2(0,T)}\right)$$
    for some positive constant $C$.
\end{proposition}

\section{Controllability and the moment problem}\label{sec3}
To begin with, we define the boundary null controllability for System \eqref{eq:1}.
\begin{definition}
System \eqref{eq:1} is boundary null controllable in $\mathcal{H}$ in time $T>0$ if for every initial datum $U_0\in \mathcal{H}$ there exists a control $f\in L^2(0,T)$ such that the corresponding solution $U$ of \eqref{eq:1} satisfies $U(T)=0$ in $\mathcal{H}$.
\end{definition}
Then, we state a first characterization of boundary null controllability.
\begin{lemma} \label{lmm:1} System \eqref{eq:1} is boundary null controllable in $\mathcal{H}$ in time $T>0$ if and only if for any initial datum $U_0 \in \mathcal{H}$ there exists $f\in L^2(0,T)$ such that the solution $U$ of \eqref{eq:1} satisfies
\begin{equation} \label{eq:6}
-(U(0),\Phi(0))_{\mathcal{H},H_1}=\int_{0}^{T}\varphi
_{x}(0,t)f(t)dt
\end{equation} 
for any $\Phi_T:=(\varphi_T,\varphi_{T,1}) \in H_1$, where $\Phi:=(\varphi,\varphi_1)$ is the solution of the backward adjoint problem given by
\begin{equation}\label{eq:2}
\begin{cases}
\varphi_{t}+\varphi_{xx}=0, \quad \text{in } D_T,\\ 
a\varphi _{xx}(1,t)-b\varphi_1(t)+d\varphi _{x}(1,t)=0, \\ 
\varphi(0,t)=0, \\ 
\varphi(\cdot,T)=\varphi_{T}, \quad \text{in } \Omega, \; \varphi(1,0)=\varphi_{T,1}.
\end{cases} 
\end{equation}
\end{lemma}
\begin{proof}
By density, we may assume that $U_0\in H$ and $f\in H^1(0,T)$. Let $U$ and $\Phi$ be the corresponding smooth solutions of systems \eqref{eq:1} and \eqref{eq:2}, respectively.
Since
\begin{eqnarray*}
\frac{d}{dt}(U,\Phi )_{H} &=&(U_{t}(\cdot,t),\Phi
(\cdot,t))_{H}+(U(\cdot,t),\Phi_{t}(\cdot,t))_{H},
\end{eqnarray*}
we can calculate the following inner product by using integration by parts
\begin{eqnarray*}
(U_{t},\Phi)_{H} &=&\int_{0}^{1}u_{xx}\varphi dx+\frac{a}{d}\varphi (1,t)\big[-\frac{d}{a}u_x(1,t)+\frac{b}{a}u(1,t)\big]
 \\
&=&\big[\varphi u_{x-}u\varphi _{x}\big]_{0}^{1}+\int_{0}^{1}u\varphi _{xx}dx+\varphi (1,t)\big[-u_x(1,t)+\frac{b}{d}u(1,t)\big].
\end{eqnarray*}
Using boundary conditions, we obtain
\begin{align*}
(U_{t},\Phi)_{H} &=\int_{0}^{1}u\varphi _{xx}dx+\varphi
(1,t)u_{x}(1,t)-\varphi (0,t)u_{x}(0,t)-u(1,t)\varphi _{x}(1,t) \\
& +u(0,t)\varphi_{x}(0,t)+ \left(\frac{b}{d}u(1,t)-u_{x}(1,t)\right)\varphi (1,t) \\
&\hspace{-1cm}=\int_{0}^{1}u\varphi _{xx}dx+\frac{b}{d}u(1,t)\varphi (1,t)-u(1,t)\varphi
_{x}(1,t)+u(0,t)\varphi_{x}(0,t)-\varphi (0,t)u_{x}(0,t) \\
&\hspace{-1cm}=\int_{0}^{1}u\varphi_{xx}dx+u(1,t)\left(\frac{b}{d}\varphi (1,t)-\varphi_{x}(1,t)\right)+u(0,t)\varphi_{x}(0,t)-\varphi (0,t)u_{x}(0,t).
\end{align*}
Since $\varphi (0,t)=0$ and $\frac{b}{d}\varphi (1,t)-\varphi _{x}(1,t)=\frac{a}{d}\varphi _{xx}(1,t),$ then 
\begin{eqnarray*}
(U_{t},\Phi)_{H} &=&\int_{0}^{1}u\varphi_{xx}dx+\frac{a}{d}\varphi
_{xx}(1,t)u(1,t)+f(t)\varphi _{x}(0,t) \\
&=&-(U,\Phi_{t})_{H}+f(t)\varphi_{x}(0,t).
\end{eqnarray*}
Therefore,
\begin{equation*}
\frac{d}{dt}(U,\Phi )_{H} =-(U,\Phi_{t})_{H}+f(t)\varphi_{x}(0,t)+(U(x,t),\Phi_{t}(x,t))_{H}.
\end{equation*}
Then we obtain,
\begin{equation*}
\frac{d}{dt}(U,\Phi)_{H}=f(t)\varphi_{x}(0,t).
\end{equation*}%
If we integrate this equality from $0$ to $T$, we obtain 
\begin{equation}\label{eq:3}
(U(\cdot,T),\Phi_T)_{H}-(U(\cdot,0),\Phi(\cdot,0))_{H}=\int_{0}^{T}\varphi_{x}(0,t)f(t)dt.
\end{equation}
If the identity \eqref{eq:6} holds, then $(U(\cdot,T),\Phi_T)_{H}=0$ is satisfied for all 
$\Phi_T \in H_1$, which implies that $U(\cdot,T)=0$ in $\mathcal{H}$. We assume in the opposite direction that the system is controllable. As a result $(U(\cdot,T),\Phi_T)_{H}=0$. By substituting this value into \eqref{eq:3}, we obtain the identity \eqref{eq:6}. This completes the proof.
\end{proof}


To reduce the null controllability into an equivalent moment problem, we need to analyze the spectral properties associated with System \eqref{eq:2}. For this reason, we will use eigenvalues and eigenfunctions of the following auxiliary spectral problem
\begin{equation}\label{eq:4}
	\begin{cases}
	&\hspace{-0.4cm} y^{\prime \prime}+\lambda y=0,\\
	&\hspace{-0.4cm} y(0)=0,\\
	&\hspace{-0.4cm} (a \lambda+b)y(1)=dy^{\prime}(1).
\end{cases}
\end{equation}
It is widely acknowledged that the primary difficulty in using the Fourier method lies in its basisness, which requires expanding solutions in terms of the eigenfunctions of the auxiliary spectral problem. Unlike the classical Sturm-Liouville problem, the problem at hand, which has a spectral parameter in the boundary condition, cannot be solved using the conventional methods of eigenfunction expansion, as outlined in works such as Tichmarsh \cite{Tich62} and Naimark \cite{Nai68}.

The spectral problem \eqref{eq:4} with $ad>0$ is examined in \cite{Kerimov:2015}. It is evident from both \cite{Kerimov:2015} and the analysis of Section \ref{sec2} that the eigenvalues, denoted as $\lambda_n$, $n=0,1,2, \ldots$, are real distinct and form an unbounded increasing sequence
$$\lambda_0 <\lambda_1<\ldots <\lambda_n<\ldots.$$
Furthermore, the corresponding eigenfunction $y_n$ associated with each eigenvalue $\lambda_n$ has exactly $n$ simple zeros in the interval $(0,1).$

Let $\mu_n$ be simple positive roots of the characteristic equation
\begin{equation}\label{car}
\left(\frac{a}{d} \mu^2+\frac{b}{d}\right) \sin \mu=\mu \cos \mu.
\end{equation}
Then, the positive eigenvalues are given by $\lambda_n=\mu_n^2.$
The arrangement of all eigenvalues with respect to $0$ depends on the parameter $\dfrac{b}{d}$ as follows
\begin{equation}\label{eq:5}
\begin{gathered}
\text{ If } \frac{b}{d}>1:\quad \lambda_0<0<\lambda_1<\lambda_2<\cdots;\\
\text{ If } \frac{b}{d}=1:\quad \lambda_0=0<\lambda_1<\lambda_2<\cdots; \\
\text{ If } \frac{b}{d}<1:\quad 0<\lambda_0<\lambda_1<\lambda_2<\cdots.
\end{gathered}
\end{equation}
For the case $\dfrac{b}{d}>1$, the eigenfunctions are given by the functions $y_0(x)=e^{\mu_0 x}-e^{-\mu_0 x},\; y_n(x)=\sin \left(\mu_n x\right)$, $\; n=1,2, \ldots$, which correspond to the eigenvalues $\lambda_0=-\mu_0^2,\; \lambda_n=\mu_n^2,\; n=1,2, \ldots,$ respectively. Moreover, it holds that $\pi n<\mu_n<\dfrac{\pi}{2}+\pi n$.

For the case $\dfrac{b}{d}=1$, the eigenfunctions are given by the functions $y_0(x)=x,\; y_n(x)=\sin \left(\mu_n x\right),\; n=$ $1,2, \ldots$, which correspond to the eigenvalues $\lambda_0=0, \;\lambda_n=\mu_n^2,\; n=1,2, \ldots$ Additionally, it holds that $\pi n<\mu_n<\dfrac{\pi}{2}+\pi n$.

For the case $\dfrac{b}{d}<1$, the eigenfunctions are given by $y_n(x)=\sin \left(\mu_n x\right),\; n=0,1,2, \ldots$, which correspond to the eigenvalues $\lambda_n=\mu_n^2, n=0,1,2, \ldots$ Furthermore, it holds that $\pi n<\mu_n<\dfrac{\pi}{2}+\pi n$ if $0 \leq \dfrac{b}{d}<1$ and $\pi n<\mu_n<\pi+\pi n$ if $\dfrac{b}{d}<0$.
By performing a comprehensive analysis of the characteristic equation, we can deduce the following  asymptotic behavior of the eigenfunctions and eigenvalues:
$$
\begin{aligned}
\mu_n & =\pi n+\frac{d}{a \pi n}+O\left(\frac{1}{n^3}\right), \\
y_n(x) & =\sin (\pi n x)+\left[\frac{d}{a \pi n} \cos (\pi n x)+O\left(\frac{1}{n^3}\right)\right] x.
\end{aligned}
$$
Therefore, the eigenelements of $\mathcal{A}_{-1}$ are 
$$Y_n:=\begin{pmatrix}
    y_n\\
    y_n(1)
\end{pmatrix}$$
    and satisfy the orthogonality relation $(Y_n,Y_m)_H=0 \text{ for } n\neq m.$
$$
$$
Since $-\mathcal{A}$ is self-adjoint and the eigenvalues $\lambda_n$ are distinct, then $Y_n$ are orthogonal in $H$ and $$Z_n:=(z_n,z_n(1))=\dfrac{Y_n}{\|Y_n\|_H}, \; n\ge 0,$$ 
form an orthonormal basis of $H$ of eigenfunctions of $-\mathcal{A}$. Moreover, $\{Z_n\}$ forms an orthogonal basis of $\mathcal{H}$ when \(\dfrac{b}{d}<1\).

To derive the solution $\Phi(x,t)$ of the problem \eqref{eq:2}, Theorem 1 in \cite{Fulton_1977} was used. According to this approach, the solution $\Phi(x,t)$ has the following form:
\begin{equation}\label{eq:9}
\Phi(x,t)=\sum_{n=0 }^{\infty}\beta_ne^{-\lambda_n(T-t)}Z_n(x), 
\end{equation}
where $\beta_n=(\Phi_T,Z_n)_H$ for $n=0,1,2,\ldots.$ In particular,
$$\varphi(x,t)=\displaystyle\sum_{n=0 }^{\infty}\beta_n e^{-\lambda_n(T-t)}z_n(x)$$
for $n=0,1,2,\ldots.$

In what follows, to state the boundary null controllability, we will focus on the case where \(\dfrac{b}{d}<1\), as indicated by the discussion in \eqref{eq:5}, where all eigenvalues are positive (see Remark \ref{rmkg}). Therefore, we can conclude the following result.


\begin{theorem}\label{thm:1} System \eqref{eq:1} is boundary null controllable in $\mathcal{H}$ in time $T>0$ if and only if for any $
	U_{0}\in \mathcal{H}$ with Fourier expansion
$U_{0}=\displaystyle\sum_{n=0 }^{\infty}\eta_n Z_n,$
	there exists a function $\theta \in L^2(0,T)$ such that
	\begin{equation}\label{eq:7}
		\int_{0}^{T} \theta(t)e^{-\lambda _{n}t}dt=-\frac{\eta_{n}e^{-\lambda _{n}T}\|Y_n\|_H}{\mu_n},\qquad
	\end{equation}
	where $\lambda_n=\mu_n^2$ and $\eta_n=(U_0,Z_n)_{\mathcal{H},H_1}$ for $n=0,1,2, \ldots.$
\end{theorem}
\begin{proof}
By Lemma \ref{lmm:1}, System \eqref{eq:1} is null controllable with a control $f(t)$ if and only if Equation \eqref{eq:6} holds. Thus, by substituting the values of $U_{0}$ and $\Phi$ into Equation \eqref{eq:6}, we obtain
$$-\sum_{n=0}^{\infty}\eta_n\beta_ne^{-\lambda_nT}=\int_{0}^{T}f(t)\sum_{n=0}^{\infty}\beta_n\dfrac{\mu_n}{\|Y_n\|_H}e^{-\lambda_n(T-t)}dt.$$
Equation \eqref{eq:6} holds if and only if it holds for $\Phi_m^T=Z_m, m=0,1,2,\ldots.$ In this case, $\beta_n=\delta_{n,m},\; m,n=0,1,2.\ldots,$ and 
$$\int_{0}^{T}f(t)\dfrac{\mu_n}{\|Y_n\|_H}e^{-\lambda_n(T-t)}dt=-\eta_n e^{-\lambda_nT}$$ for $n=0,1,2, \ldots.$ By replacing $T-t$ by $t$ in the last integral and defining $\theta(t):=f(T-t)$, the proof is completed.
\end{proof}

Note that the control function, denoted as $f(t),$ is determined by choosing the function $\theta(t)$ that satisfies Equation \eqref{eq:7}. This is a moment problem in $L^2(0,T)$ with respect to the family $\Lambda=\{e^{-\lambda_mt}\}_{m\geq0}$. Suppose that we can
construct a sequence $\{\Theta_n\}_{n\geq0}$ of functions biorthogonal to the set $%
\Lambda $ in $L^2(0,T)$ such that
\begin{equation*}
	\int_{0}^{T}e^{-\lambda_mt}\Theta_n(t)dt=\delta_{n,m}
\end{equation*}
for all $m,n=0,1,2\ldots.$ Then the moment problem \eqref{eq:7} has a (formal) solution by setting
\begin{equation}\label{eq:8}
	\theta(t)=\sum_{n=0}^{\infty} -\frac{\eta_{n}e^{-\lambda _{n}T}\|Y_n\|_H}{\mu_n}\Theta_{n}(t).
\end{equation}
From \cite{Kerimov:2015}, it is established that a detailed examination of the characteristic equation \eqref{car} enables us to determine the following asymptotic behavior of the (square roots of) eigenvalues
\begin{equation}\label{asy}
    \mu_n=\pi n+\frac{d}{a \pi n}+O\left(\frac{1}{n^3}\right).
\end{equation}
Now the existence of the biorthogonal sequence $\{\Theta_n\}_{n\geq0}$ is guaranteed by Müntz theorem, see e.g. \cite[p.~54]{Sc43}. Indeed, since by equality \eqref{asy} we have $\displaystyle\sum_{n=0 }^{\infty} \frac{1}{\lambda_n} <\infty$, $\lambda_n=\mu_n^2,$ Müntz theorem implies that $\Lambda$ is not total in $L^2(0,T)$. Then it admits a biorthogonal sequence $\{\Theta_n\}_{n\geq0}$.

To determine a solution of \eqref{eq:7}, we need to show the convergence of the series presented in Equation \eqref{eq:8}. To this end, we provide the following remark.
\begin{remark}
From \eqref{asy}, it is easy to show that the sequence
$$ \{\lambda_n=\mu_n^2, n=0,1,2\ldots\}$$
	satisfies the conditions of Theorem 1.1 and Theorem 1.5 in \cite{Fattorini:74}. Hence, for any $\epsilon > 0$ there exists a positive constant $K_\epsilon$ such that 
	$$||\Theta_n||_{L^2(0,T)} < K_\epsilon e^{\epsilon\lambda_n},\; n = 0, 1,\ldots.$$
 This estimate allows one to prove (as shown in \cite{AmmarKhodja:2017} or \cite{Boyer:2020}) that the biorthogonal series \eqref{eq:8} converges absolutely, so the moment problem defined in Theorem \ref{thm:1} has a solution $\theta\in L^2(0, T)$.
\end{remark}
Therefore, we deduce the following result.
\begin{proposition}\label{pctr}
For any initial datum $U_0\in \mathcal{H}$ with Fourier expansion $\displaystyle U_{0}=\sum\limits_{n=0 }^{\infty}\eta_n Z_n,$ there exists a null control $f \in L^2(0, T)$ of System \eqref{eq:1} given by
$$f(t)=\sum_{n=0 }^{\infty} -\frac{\eta_{n}e^{-\lambda _{n}T}\|Y_n\|_H}{\mu_n}\Theta_{n}(T-t).$$
\end{proposition}
\begin{remark}\label{rmkg}
We assumed at the outset that the eigenvalues include only positive real numbers $\lambda_n>0$ (that is, $\dfrac{b}{d}<1$). However, some of the eigenvalues of \eqref{eq:4} may be negative or null in the cases where $\dfrac{b}{d}\geq 1$. In such cases, we still have similar boundary null controllability results. Indeed, by setting 
$$\theta(t)=e^{-\lambda t}\tilde{\theta}(t)$$
Equation \eqref{eq:7} becomes $$\int_0^Te^{-(\lambda_n+\lambda)t}\tilde{\theta}(t)dt=-\frac{\eta_{n}e^{-\lambda _{n}T}\|Y_n\|_H}{\mu_n}$$ and $$\lambda_n+\lambda>0$$ for all $n$ if $\lambda$ is chosen so that $\lambda>-\lambda_0$.
\end{remark}

The theoretical controllability results presented above are obtained by the moment method which draws on the the existence of a biorthogonal sequence associated with the exponential family of eigenvalues $\Lambda$. The construction of such a sequence is quite complicated; this is why we will consider an alternative numerical approach.

\section{Numerical algorithm}\label{sec5}
In this section, we present a numerical algorithm for computing the approximate boundary controls of minimal $L^2$-norm by the penalized HUM approach combined with a CG scheme. We refer to \cite{GL'08} for the details. Note that the numerical analysis of the boundary controllability problem we consider here is delicate and the convergence analysis is not covered by the abstract framework in \cite{Bo13}.

\subsection{The HUM controls}
Next, we will design an algorithm inspired by \cite[Chapter 2]{GL'08} for Dirichlet boundary conditions.

Let us assume without loss of generality that the target function $U_T:=(u_T,u_{T,1})=(0,0)$ (for simplicity), otherwise we can consider any $U_T\in \mathcal{H}$. Let $\varepsilon>0$ be a fixed penalization parameter. We define the cost functional $J_{\varepsilon}: L^2(0,T) \longrightarrow \mathbb{R}$ to be minimized on $L^2(0,T)$ as follows
\begin{equation*}
J_{\varepsilon}\left(f\right)=\frac{1}{2}\int_0^T |f(t)|^2 dt +\frac{1}{2\varepsilon}\left\|U(\cdot, T)\right\|^2_{\mathcal{H}},
\end{equation*}
where $U$ is the solution of \eqref{eq:1} corresponding to the control $f$. Note that $J_\varepsilon$ is strictly convex, continuous and coercive.
\begin{remark}
In what follows, we only consider HUM approximate controls $f_\varepsilon$. We expect that HUM null controls can be obtained by carefully studying the limit of $f_\varepsilon$ as $\varepsilon\to 0$ in $L^2(0,T)$. We refer to \cite{Bo13} for the distributed control case.
\end{remark}
Calculating the increments of $J_\varepsilon$ by the adjoint methodology, we show that the unique minimizer $f:=f_\varepsilon \in L^2(0, T)$ of $J_{\varepsilon}$ is characterized by the following Euler-Lagrange equation
\begin{equation}\label{Eq2}
(J'_\varepsilon(f),g)_{L^2(0,T)}=\int_0^T (f(t)+p_x(0,t))g dt = 0 \qquad \forall f,g \in L^2(0,T),
\end{equation}
where $p$ is the solution of 
\begin{empheq}[left = \empheqlbrace]{alignat=2}
\begin{aligned}
&p_{t}+p_{xx}=0 \quad \text{in } D_T, \\ 
&a p_{xx}(1,t)-bp(1,t)+d p_{x}(1,t)=0, \\ 
&p(0,t)=0, \\ 
&p(x,T)=h(x) \quad \text{ in }\Omega_1, \label{opt}
\end{aligned}
\end{empheq}
with $h$ solves
\begin{equation}
    \begin{cases}h_{xx}=\alpha h - \varepsilon^{-1} u(x,T) & \text { in } \Omega, \\ a h_{xx}(1)-b h(1)+d h_x(1)=0, &  \\ h(0)=0. & \end{cases}
\end{equation}
The optimality condition \eqref{Eq2} implies that the optimal control is given by
$$f(t)=-p_x(0,t) \qquad \forall t\in (0,T).$$

We introduce the continuous operator $\Gamma: H_1 \to \mathcal{H}$
given by $\Gamma V=-Y(\cdot, T)$ for all $V\in H_1$, where $Y(\cdot, T)$ is defined as follows: we first solve the problem
\begin{empheq}[left = \empheqlbrace]{alignat=2}
\begin{aligned}
&p_{t}+p_{xx}=0 \quad \text{in } D_T, \\ 
&a p_{xx}(1,t)-bp(1,t)+d p_{x}(1,t)=0, \\ 
&p(0,t)=0, \\ 
&p(x,T)=v(x)\; \text{ in }\Omega, \; p(1,T)=v_1, \label{solp}
\end{aligned}
\end{empheq}
and then we solve
\begin{empheq}[left = \empheqlbrace]{alignat=2}
\begin{aligned}
&y_{t}-y_{xx}=0 \quad \text{in } D_T, \\ 
&a y_{xx}(1,t)-by(1,t)+d y_{x}(1,t)=0, \\ 
&y(0,t)=-p_x(0,t), \\ 
&y(x,0)=u_0(x) \; \text{ in }\Omega,\; y(1,0)=u_{0,1}. \notag
\end{aligned}
\end{empheq}
A simple calculation using \eqref{eq:3} yields
$$(\Gamma V_1,V_2)_H=\int_0^T p^1_x(0,t)p^2_x(0,t) dt \qquad \forall V_1,V_2\in H_1,$$
with $p^k$ is the solution to \eqref{solp} corresponding to $V_k$, $k=1,2$. Then we can deduce that $\Gamma$ is self-adjoint and positive semi-definite. Let $\alpha\in \rho(\mathcal{A}_{-1})$; e.g., we can choose $\alpha=0$ if $\dfrac{b}{d}\neq 1$ and $\alpha=-1$ if $\dfrac{b}{d}= 1$ to guarantee that the following system admits a unique solution. Let $G:=(g,g_1)\in H_1$ be the unique solution of the elliptic problem
$$\begin{cases}g_{xx}=\alpha g - \varepsilon^{-1} y(x,T) & \text { in } \Omega, \\ a g_{xx}(1)-bg(1)+d g_x(1)=0, &  \\ g(0)=0. & \end{cases}$$
Therefore, $G\in H_1$ is the unique solution of
$$\varepsilon(\alpha -\mathcal{A}_{-1})G +\Gamma G=0\qquad \text{in } \mathcal{H}.
$$
This (dual) problem can be solved by a conjugate gradient
method as follows.
\smallskip

\noindent\rule{\textwidth}{1pt}
\textbf{Algorithm 1:} HUM with CG Algorithm

\noindent\rule{\textwidth}{1pt}
\textbf{1} Set $k=0$ and choose an initial guess $V^0 \in H_2$.\\
\textbf{2} Solve the problem
\begin{empheq}[left = \empheqlbrace]{alignat=2} \label{adj}
\begin{aligned}
&p^0_{t}+p^0 _{xx}=0 \quad \text{in } D_T, \\ 
&a p^0_{xx}(1,t)-b p^0(1,t)+d p^0_{x}(1,t)=0, \\ 
&p^0(0,t)=0, \\ 
&p^0(x,T)=v^0(x), \; p^0(1,T)=v^0_1
\end{aligned}
\end{empheq}
and set $f^0(t)=-p^0_x(0,t)$. \\ 
\textbf{3} Solve the problem
\begin{empheq}[left = \empheqlbrace]{alignat=2}
\begin{aligned}
&y^0_{t}-y^0_{xx}=0 \quad \text{in } D_T, \\ 
&a y^0_{xx}(1,t)-b y^0(1,t)+d y^0_{x}(1,t)=0, \\ 
&y^0(0,t)=f^0(t), \\ 
&y^0(x,0)=u_0(x), \; y^0(1,0)=u_{0,1},\notag \; 
\end{aligned}
\end{empheq}
then solve 
\begin{empheq}[left = \empheqlbrace]{alignat=2}
\begin{aligned}
&g^0_{xx}-\alpha g^0=\varepsilon \left(v^0_{xx}-\alpha v^0\right)+y^0(T,x) \quad \text{in } \Omega, \\ 
&a g^0_{xx}(1)-b g^0(1)+d g^0_{x}(1)=0, \\ 
&g^0(0)=0, \notag
\end{aligned}
\end{empheq}
and set $W^0=G^0$.\\
\textbf{4} For $k=1,2,\ldots,$ until convergence,
solve the problem
\begin{empheq}[left = \empheqlbrace]{alignat=2}
\begin{aligned}
&p^k_{t}+p^k_{xx}=0 \quad \text{in } D_T, \\ 
&a p^k_{xx}(1,t)-b p^k (1,t)+d p^k_{x}(1,t)=0, \\ 
&p^k(0,t)=0, \\ 
&p^k(x,T)=w^{k-1}(x), \; p^k(1,T)=w^{k-1}_1 \notag
\end{aligned}
\end{empheq}
and set $f^k(t)=-p^k_x(0,t)$.\\
\textbf{5} Solve the problem
 \begin{empheq}[left = \empheqlbrace]{alignat=2}
\begin{aligned}
&y^k_{t}-y^k_{xx}=0 \quad \text{in } D_T, \\ 
&a y^k_{xx}(1,t)-b y^k(1,t)+d y^k_{x}(1,t)=0, \\ 
&y^k(0,t)=f^k(t), \\ 
&y^k(x,0)=u_0(x), \; y^k(1,0)=u_{0,1},\notag \; 
\end{aligned}
\end{empheq}
then solve
\begin{empheq}[left = \empheqlbrace]{alignat=2}
\begin{aligned}
&\bar{g}^k_{xx}-\alpha \bar{g}^k=\varepsilon \left(w^{k-1}_{xx}-\alpha w^{k-1}\right)+y^k(T,x) \quad \text{in } \Omega, \\ 
&a \bar{g}^k_{xx}(1)-b \bar{g}^k(1)+d \bar{g}^k_{x}(1)=0, \\ 
&\bar{g}^k(0)=0, \notag
\end{aligned}
\end{empheq}
and compute $\rho_k=\dfrac{\|g^{k-1}_x\|^2_H}{(\bar{g}^k_x, w^{k-1}_x)_H},$
then
$$v^k=v^{k-1}-\rho_k w^{k-1} \qquad \text{ and } \qquad g^k=g^{k-1}-\rho_k \bar{g}^k. \;$$
\textbf{If} $\dfrac{\|g^{k}_x\|_H}{\|g^{0}_x\|_H} \le tol$, stop the algorithm, set $v=v^{k}$ and solve \eqref{adj} with $v^0=v$
to get $f(t)=-p^0_x(0,t)$. \newline
\textbf{Else} compute $$\gamma_k=\dfrac{\|g^{k}_x\|^2_H}{\|g^{k-1}_x\|^2_H} \qquad \text{ and then } \qquad w^k=g^{k}+\gamma_k w^{k-1}.$$
\noindent\rule{\textwidth}{1pt}

\subsection{Numerical experiments}
We employ the method of lines to discretize and solve different PDEs in Algorithm 1 along the spatial interval $[0,1]$. We use the uniform grid $x_j=j \Delta x$,  $j=\overline{0,N_x}$, where $\Delta x=\dfrac{1}{N_x}$ as a spatial mesh parameter. Setting $u_j(t):=u(x_j,t)$, we use the standard centered difference approximation of $u_{xx}$:
$$u_{xx}(x_j,t) \approx \frac{u_{j-1}(t)-2 u_j(t) + u_{j+1}(t)}{(\Delta x)^2}, \quad j=\overline{1, N_x-1}.$$
The boundary derivatives can be approximated using backward formulas:
\begin{align*}
    u_x(1,t) &\approx \frac{u_{N_x}(t)-u_{N_x-1}(t)}{\Delta x},\\
    u_{xx}(1,t) &\approx \frac{u_{N_x-2}(t)-2 u_{N_x-1}(t) + u_{N_x}(t)}{(\Delta x)^2}.
\end{align*}
It suffices to solve the resulting system of ordinary differential equations. The method of lines has been performed in the \texttt{Wolfram} language by adopting the function \texttt{NDSolve}. 

In all numerical tests, we take the following values for our computations: $N_x=25$ for the mesh parameter, $T=1$ for the terminal time, and $U_0:=(u_0,u_{0,1})$ given by $u_0(x)=\sqrt{2} \sin(\pi x), \; x \in [0,1],$ $u_{0,1}=0$, 
is the initial datum to be controlled. In Algorithm 1, we take $\varepsilon=10^{-3}$ as penalization parameter. Furthermore, for the parameter $\alpha$, we take $\alpha=0$ if $\dfrac{b}{d}\neq 1$ ($0$ is not an eigenvalue) and $\alpha=-1$ if $\dfrac{b}{d}= 1$ ($-1$ is not an eigenvalue) so that $\alpha\in \rho(\mathcal{A}_{-1})$ and the corresponding systems admit unique solutions. We also take the initial guess as $V^0=0$. To perform a comparative analysis with respect to the parameter $\dfrac{b}{d}$, we fix the number of iterations at $n^\mathrm{iter}=7$ for the plots.

\subsubsection*{\bf (i) Case $\dfrac{b}{d}<1$:}
Here we take $a=b=1$, $d=3$ and $\alpha=0$. We plot the uncontrolled and controlled solutions of $\eqref{eq:1}$.
\begin{figure}[H]
\centering
\begin{minipage}{0.45\textwidth}
\includegraphics[scale=0.4]{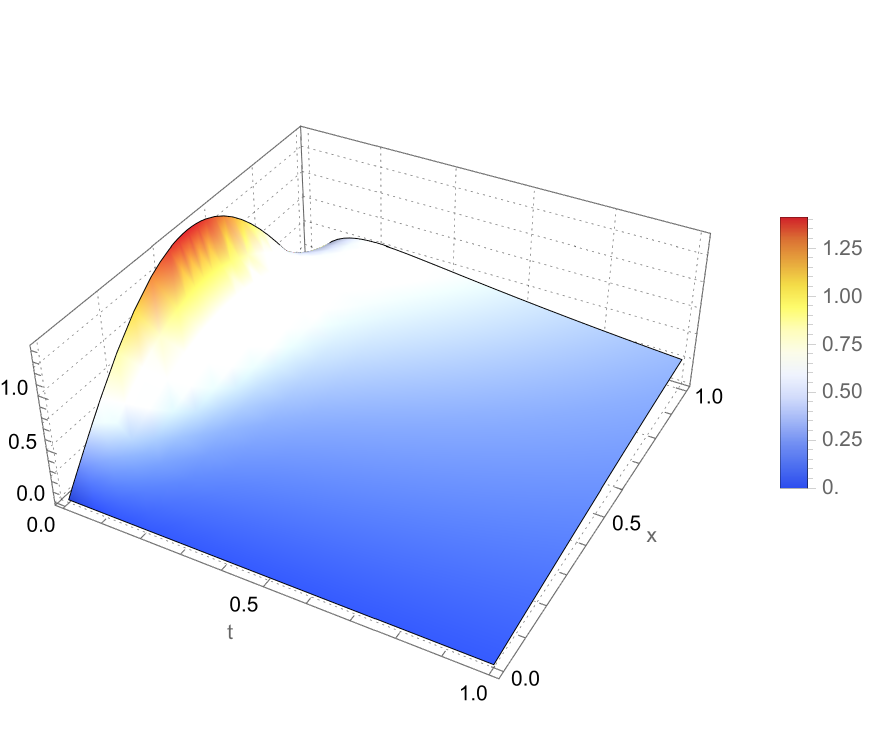}
\end{minipage}\hfill
\begin{minipage}{0.45\textwidth}
\centering
\includegraphics[scale=0.4]{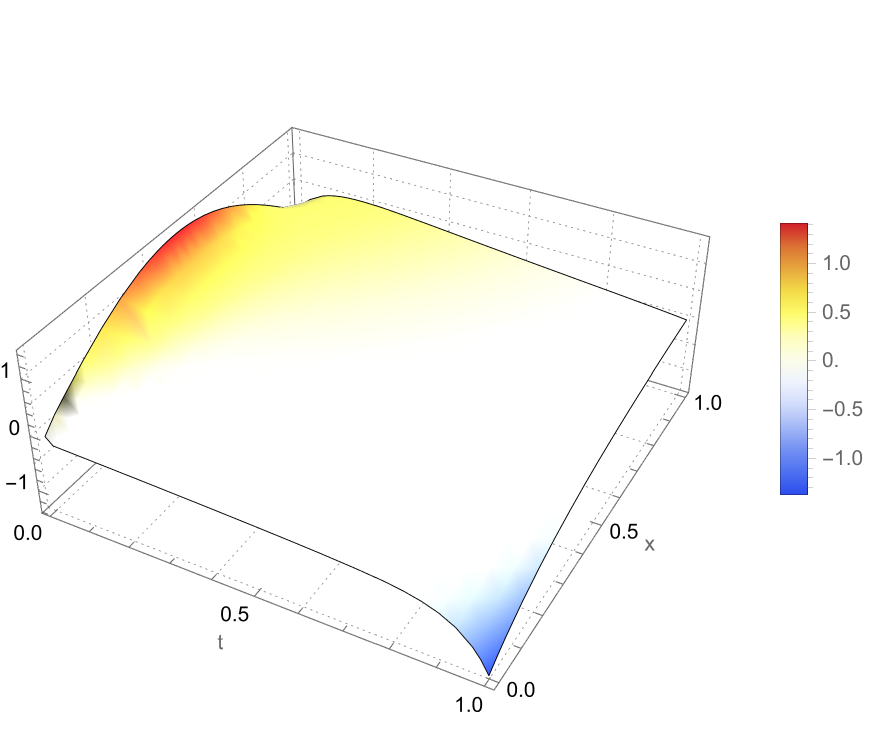}
\end{minipage}
\caption{The uncontrolled solution (left) and the controlled solution (right) in the case $\dfrac{b}{d}<1$.}
\label{fig10}
\end{figure}
Next, we plot the computed control.
\begin{figure}[H]
\centering
\includegraphics[scale=0.4]{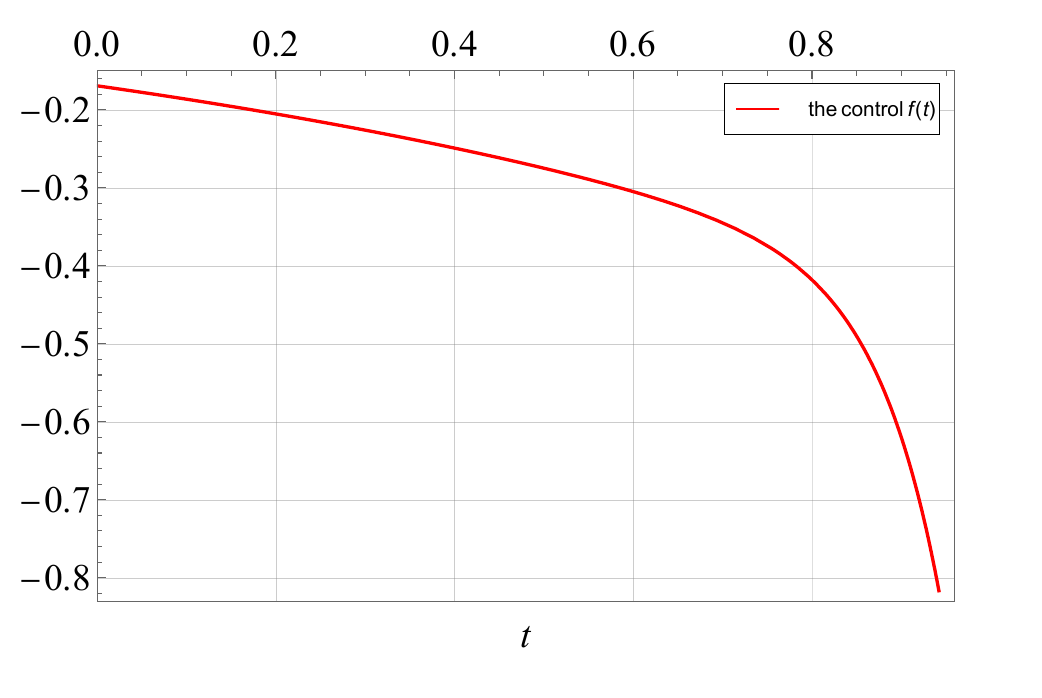}
\caption{The computed control in the case $\dfrac{b}{d}<1$.} \label{fig12}
\end{figure}

\subsubsection*{\bf (ii) Case $\dfrac{b}{d}=1$:}
Here we take $a=b=d=1$ and $\alpha=-1$. We plot the uncontrolled and controlled solutions of $\eqref{eq:1}$.
\begin{figure}[H]
\centering
\begin{minipage}{0.45\textwidth}
\includegraphics[scale=0.4]{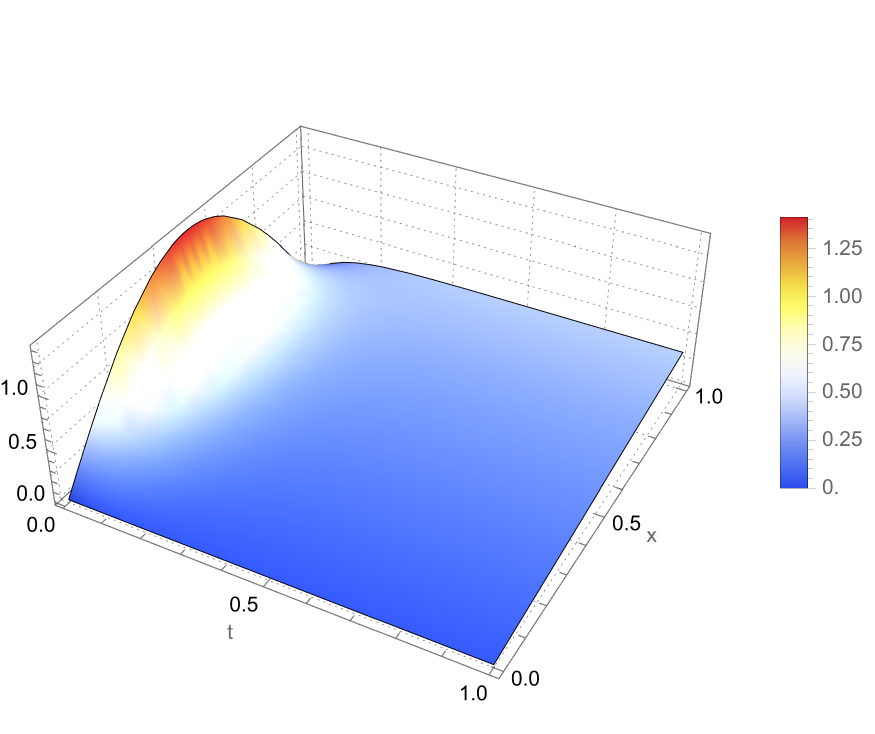}
\end{minipage}\hfill
\begin{minipage}{0.45\textwidth}
\centering
\includegraphics[scale=0.4]{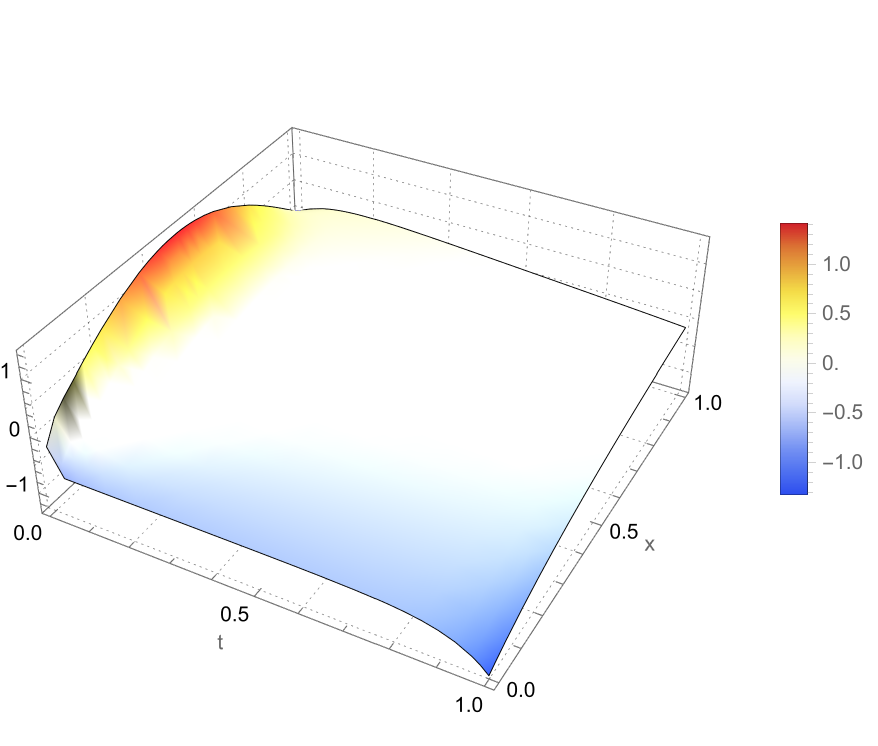}
\end{minipage}
\caption{The uncontrolled solution (left) and the controlled solution (right) in the case $\dfrac{b}{d}=1$.}
\label{fig20}
\end{figure}
Next, we plot the computed control.
\begin{figure}[H]
\centering
\includegraphics[scale=0.4]{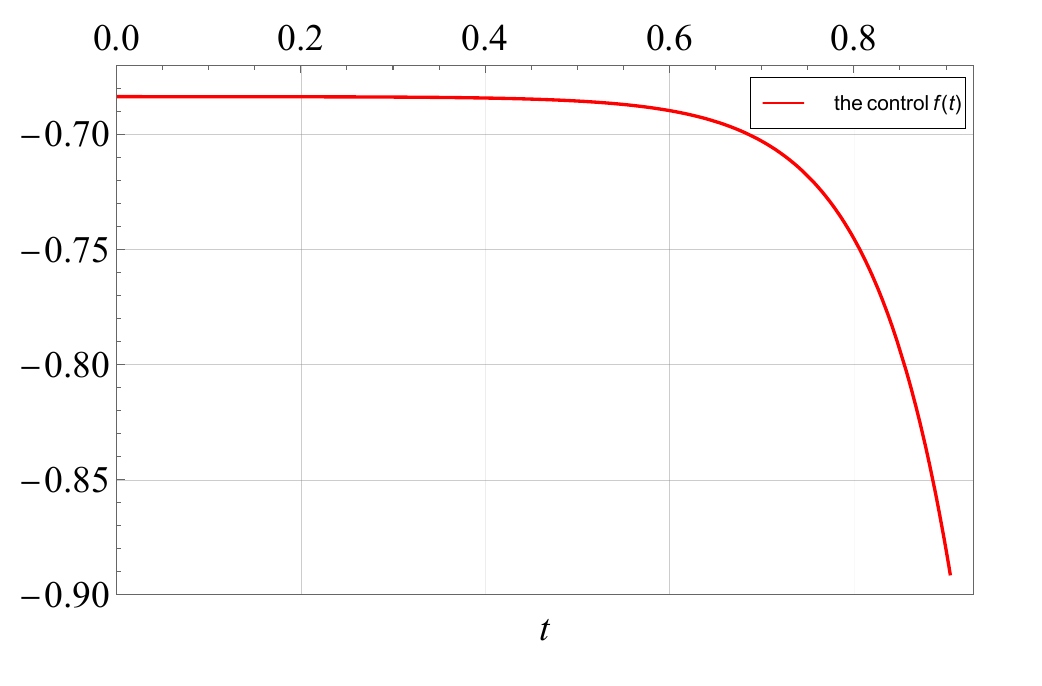}
\caption{The computed control in the case $\dfrac{b}{d}=1$.}\label{fig22}
\end{figure} 

\subsubsection*{\bf (iii) Case $\dfrac{b}{d}>1$:}
Here we take $a=1$, $b=3$, $d=1$ and $\alpha=0$. We plot the uncontrolled and controlled solutions of $\eqref{eq:1}$.
\begin{figure}[H]
\centering
\begin{minipage}{0.45\textwidth}
\includegraphics[scale=0.4]{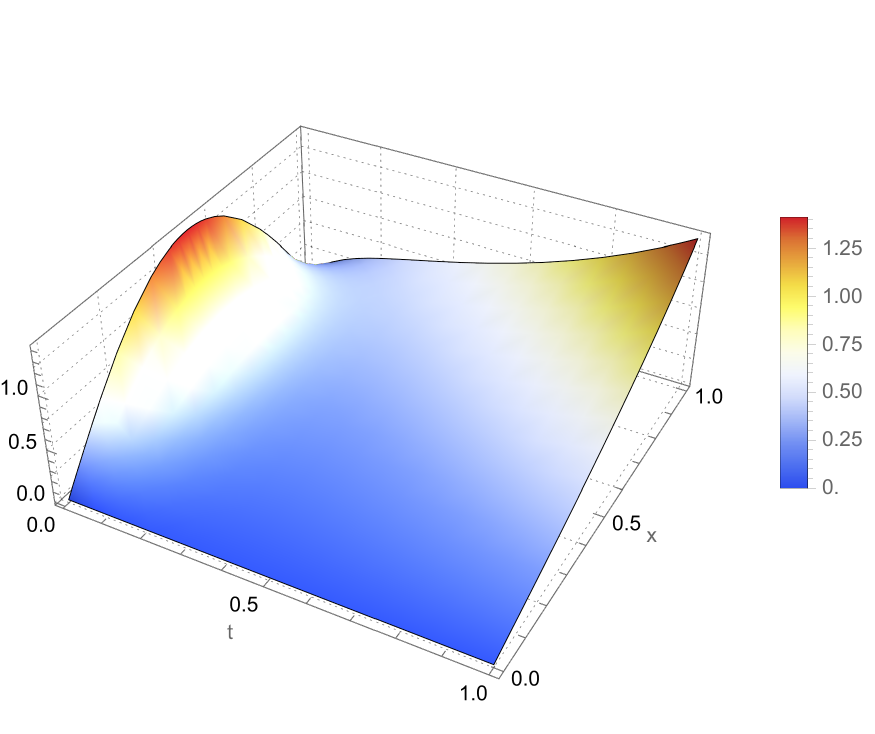}
\end{minipage}\hfill
\begin{minipage}{0.45\textwidth}
\centering
\includegraphics[scale=0.4]{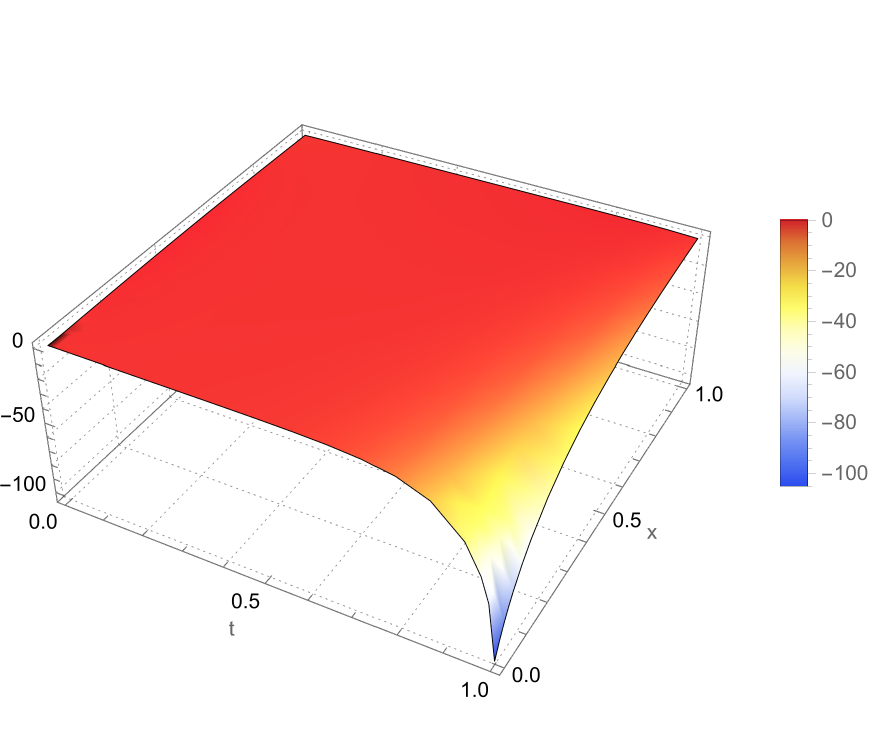}
\end{minipage}
\caption{The uncontrolled solution (left) and the controlled solution (right) in the case $\dfrac{b}{d}>1$.}
\label{fig30}
\end{figure}
Next, we plot the computed control.
\begin{figure}[H]
\centering
\includegraphics[scale=0.4]{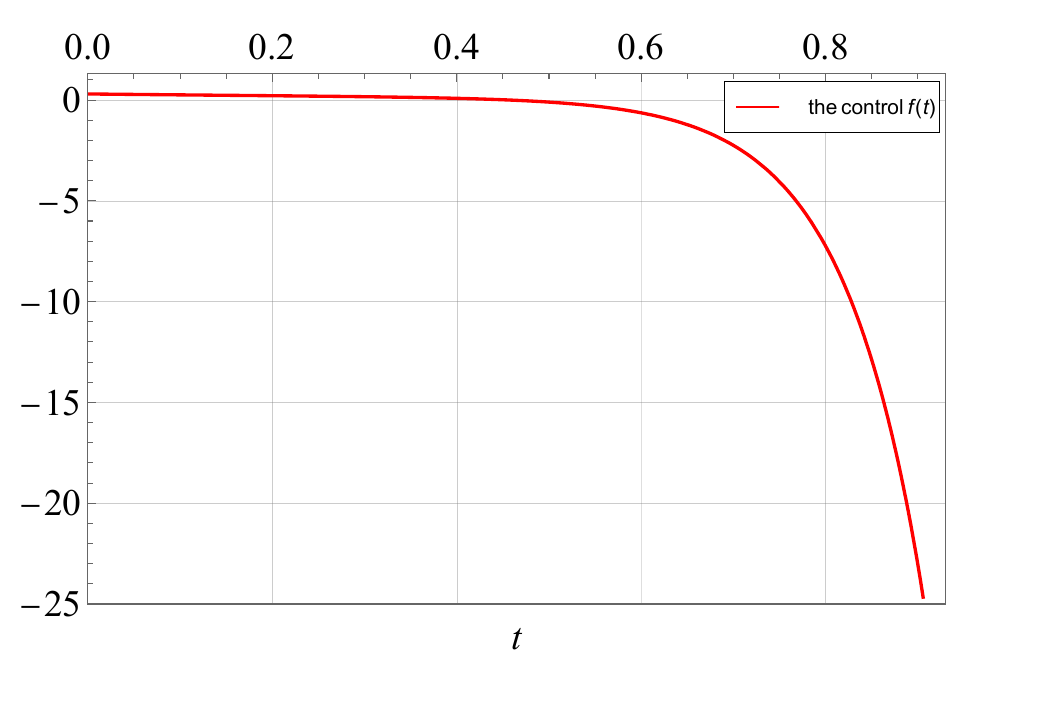}
\caption{The computed control in the case $\dfrac{b}{d}>1$.}\label{fig32}
\end{figure}

The following remarks are in order concerning the above experiments:
\begin{itemize}
    \item From Figures \ref{fig10}, \ref{fig20} and \ref{fig30}, we observe the effect of computed controls on the solutions at time $T=1$. Indeed, we see that the negative values of the solutions are concentrated near the control location $x=0$.
    \item From Figures \ref{fig12}, \ref{fig22} and \ref{fig32}, we see that the calculated controls take negative values, deriving the positive values of the solutions toward zero. We also observe that the calculated control decreases faster in the case $\dfrac{b}{d}<1$ than the case $\dfrac{b}{d}=1$, but with closer minimums. However, the case $\dfrac{b}{d}>1$ requires a much smaller minimum. This is natural when the shape of the solution in each case is taken into account.
    \item Compared to distributed control algorithms, see, e.g. \cite{Ch23}, the boundary control Algorithm 1 is slower and requires more iterations to converge (e.g., thousands of iterations for a tolerance $tol=10^{-3}$). This is due to the difficulty of the boundary control problem where the control acts as a trace on a single point $x=0$, while the distributed control acts on an interval. Moreover, Algorithm 1 contains an extra step at each iteration to solve an elliptic problem. This is not the case for distributed control problems.
\end{itemize}
\section{Conclusion}\label{sec6}
The paper considers the linear heat equation with a Wentzell-type boundary condition and a Dirichlet control. After proving its well-posedness, the control problem is reduced to a moment problem. Then, using the spectral analysis of the associated spectral problem with boundary conditions that include a spectral parameter and applying the moment method, the boundary null controllability of the system at any positive time $T$ is established. Finally, minimum energy controls are approximated by a penalized HUM approach.
The similar Dirichlet control problem for the 2D heat equation in a rectangular domain with a Wentzell boundary condition is a line of future investigation. The associated spectral problem, existence and uniqueness results have been recently studied in \cite{Ismailov:2023}.
\bigskip

\noindent\textbf{Acknowledgement.}
This study was supported by Scientific and Technological Research Council of Turkey (TUBITAK) under the Grant Number 123F231. The authors thank TUBITAK for their support.

\newpage
\bibliographystyle{plain}
\bibliography{dynamic.bib}

\begin{thebibliography}{10}

\bibitem{AmmarKhodja:2017}
Farid Ammar-Khodja.
\newblock {\em Controllability of Parabolic Systems: The Moment Method}.
\newblock Evolution Equations, Cambridge University Press, 2017.

\bibitem{AT12}
Wolfgang Arendt and A.~F. M.~Tom ter Elst.
\newblock From forms to semigroups.
\newblock In {\em Spectral Theory, Mathematical System Theory, Evolution
  Equations, Differential and Difference Equations, Oper. Theory Adv. App.},
  pages 47--69. Springer, 2012.

\bibitem{Bo13}
Franck Boyer.
\newblock On the penalised {HUM} approach and its applications to the numerical
  approximation of null-controls for parabolic problems.
\newblock In {\em ESAIM: Proceedings}, volume~41, pages 15--58. EDP Sciences,
  2013.

\bibitem{Boyer:2020}
Franck Boyer.
\newblock {Controllability of linear parabolic equations and systems}.
\newblock hal-02470625, 2020.

\bibitem{Ch23}
Salah-Eddine Chorfi, Ghita El~Guermai, Lahcen Maniar, and Walid Zouhair.
\newblock Logarithmic convexity and impulsive controllability for the
  one-dimensional heat equation with dynamic boundary conditions.
\newblock {\em IMA Journal of Mathematical Control and Information},
  39(3):861--891, 2022.

\bibitem{Chorfi:2023}
Salah-Eddine Chorfi, Ghita~El Guermai, Abdelaziz Khoutaibi, and Lahcen Maniar.
\newblock Boundary null controllability for the heat equation with dynamic
  boundary conditions.
\newblock {\em Evolution Equations and Control Theory}, 12(2):542--566, 2023.

\bibitem{EN00}
Klaus-Jochen Engel and Rainer Nagel.
\newblock {\em One-Parameter Semigroups for Linear Evolution Equations}, volume
  194.
\newblock Springer, New York, 2000.

\bibitem{Farkas:2011}
József~Z Farkas and Peter Hinow.
\newblock Physiologically structured populations with diffusion and dynamic
  boundary conditions.
\newblock {\em Mathematical Biosciences and Engineering}, 8(2):503--513, 2011.

\bibitem{Fattorini:71}
Hector~O. Fattorini and David~L. Russell.
\newblock Exact controllability theorems for linear parabolic equations in one
  space dimension.
\newblock {\em Arch. Rational Mech. Anal.}, 43:272–292, 1971.

\bibitem{Fattorini:74}
Hector~O. Fattorini and David~L. Russell.
\newblock Uniform bounds on biorthogonal functions for real exponentials with
  an application to the control theory of parabolic equations.
\newblock {\em Quarterly of Applied Mathematics}, 32:45--69, 1974.

\bibitem{Fe57}
William Feller.
\newblock Generalized second order differential operators and their lateral
  conditions.
\newblock {\em Illinois journal of mathematics}, 1(4):459--504, 1957.

\bibitem{Ca10}
Enrique Fern{\'a}ndez-Cara, Manuel Gonz{\'a}lez-Burgos, and Luz de~Teresa.
\newblock Boundary controllability of parabolic coupled equations.
\newblock {\em Journal of Functional Analysis}, 259:1720--1758, 2010.

\bibitem{Cara:2006}
Enrique Fern{\'{a}}ndez-Cara, Manuel Gonz{\'{a}}lez-Burgos, Sergio Guerrero,
  and Jean-Pierre Puel.
\newblock Null controllability of the heat equation with boundary fourier
  conditions: the linear case.
\newblock {\em {ESAIM}: Control, Optimisation and Calculus of Variations},
  12(3):442--465, 2006.

\bibitem{Fulton_1977}
Charles~T. Fulton.
\newblock Two-point boundary value problems with eigenvalue parameter contained
  in the boundary conditions.
\newblock {\em Proceedings of the Royal Society of Edinburgh: Section A
  Mathematics}, 77(3–4):293–308, 1977.

\bibitem{GL'08}
Roland Glowinski, Jacques-Louis Lions, and Jiwen He.
\newblock {\em Exact and Approximate Controllability for Distributed Parameter
  Systems: a Numerical Approach}, volume 117.
\newblock Cambridge University Press, Cambridge, UK; New York, 2008.

\bibitem{Goldstein:2006}
Gisèle~Ruiz Goldstein.
\newblock Derivation and physical interpretation of general boundary
  conditions.
\newblock {\em Adv. Differential Equations}, 11(4):457–480, 2006.

\bibitem{Go19}
Nikita~S Goncharov and Georgy~A Sviridyuk.
\newblock The heat conduction model involving two temperatures on the segment
  with wentzell boundary conditions.
\newblock {\em Journal of Physics: Conference Series}, 1352(1):012022, 2019.

\bibitem{Guo:95}
Yung-Jen.~L. Guo and Walter Littman.
\newblock Null boundary controllability for semilinear heat equations.
\newblock {\em Appl. Math. Optim.}, 32(3):281--316, 1995.

\bibitem{Emanuilov:95}
Oleg~Yu. Imanuvilov.
\newblock Controllability of parabolic equations.
\newblock {\em Mat. Sb}, 186(6):109--32, 1995.

\bibitem{Ismailov:2023a}
Mansur~I. Ismailov and Isil Oner.
\newblock Null boundary controllability for some biological and chemical
  diffusion problems.
\newblock {\em Evolution Equations and Control Theory}, 12(5):1287--1299, 2023.

\bibitem{Ismailov:2023}
Mansur~I. Ismailov and Önder Türk.
\newblock Direct and inverse problems for a 2{D} heat equation with a
  {D}irichlet–{N}eumann–{W}entzell boundary condition.
\newblock {\em Communications in Nonlinear Science and Numerical Simulation},
  127:107519, 2023.

\bibitem{Kerimov:2015}
Nazim~B. Kerimov and Mansur~I. Ismailov.
\newblock Direct and inverse problems for the heat equation with a dynamic-type
  boundary condition.
\newblock {\em {IMA} Journal of Applied Mathematics}, 80(5):1519--1533, 2015.

\bibitem{Khoutaibi:2020}
Abdelaziz Khoutaibi and Lahcen Maniar.
\newblock Null controllability for a heat equation with dynamic boundary
  conditions and drift terms.
\newblock {\em Evolution Equations {\&} Control Theory}, 9(2):535--559, 2020.

\bibitem{Khoutaibi:2022}
Abdelaziz Khoutaibi, Lahcen Maniar, and Omar Oukdach.
\newblock Null controllability for semilinear heat equation with dynamic
  boundary conditions.
\newblock {\em Discrete and Continuous Dynamical Systems - S}, 15(6):1525,
  2022.

\bibitem{Kumpf:2004}
Michael Kumpf and Gregor Nickel.
\newblock Dynamic boundary conditions and boundary control for the
  one-dimensional heat equation.
\newblock {\em Journal of Dynamical and Control Systems}, 10(2):213--225, 2004.

\bibitem{Langer:1932}
Rudolph Langer.
\newblock A problem in diffusion or in the flow of heat for a solid in contact
  with a fluid.
\newblock {\em Tohoku Mathematical Journal}, 35:260--275, 1932.

\bibitem{Maniar:2017}
Lahcen Maniar, Martin Meyries, and Roland Schnaubelt.
\newblock Null controllability for parabolic equations with dynamic boundary
  conditions.
\newblock {\em Evolution Equations {\&} Control Theory}, 6(3):381--407, 2017.

\bibitem{TW09}
Tucsnack Marius and Weiss George.
\newblock {\em Observation and Control for Operator Semigroups}, volume 194.
\newblock Birkhäuser, 2009.

\bibitem{Martin:2016}
Philippe Martin, Lionel Rosier, and Pierre Rouchon.
\newblock Null controllability of one-dimensional parabolic equations by the
  flatness approach.
\newblock {\em SIAM Journal on Control and Optimization}, 54(1):198--220, 2016.

\bibitem{Mercado:2023}
Alberto Mercado and Roberto Morales.
\newblock Exact controllability for a schrödinger equation with dynamic
  boundary conditions.
\newblock {\em SIAM Journal on Control and Optimization}, 61(6):3501--3525,
  2023.

\bibitem{Nai68}
Mark~Aronovitch Naimark.
\newblock {\em Linear Differential Operators: Elementary Theory of Linear
  Differential Equations}.
\newblock Harrap, 1968.

\bibitem{Oner:2023}
Isil Oner.
\newblock Null controllability for a heat equation with periodic boundary
  conditions.
\newblock {\em U.P.B. Sci.Bull. Series A}, 83(4):13--22, 2021.

\bibitem{Ou'05}
El-Maati Ouhabaz.
\newblock {\em Analysis of Heat Equations on Domains}.
\newblock Princeton University Press, 2009.

\bibitem{Rousseau:2007}
J{\'{e}}r{\^{o}}me~Le Rousseau.
\newblock Carleman estimates and controllability results for the
  one-dimensional heat equation with {BV} coefficients.
\newblock {\em Journal of Differential Equations}, 233(2):417--447, 2007.

\bibitem{Sc43}
Laurent Schwartz.
\newblock {\em {\'E}tude des sommes d'exponentielles r{\'e}elles}.
\newblock PhD thesis, Université de Clemont-Ferrand, 1943.

\bibitem{Elst:2013}
A.~F. M.~Tom ter Elst, Martin Meyries, and Joachim Rehberg.
\newblock Parabolic equations with dynamical boundary conditions and source
  terms on interfaces.
\newblock {\em Annali di Matematica Pura ed Applicata (1923 -)},
  193(5):1295--1318, 2013.

\bibitem{Tich62}
Edward~Charles Tichmarsh.
\newblock {\em Eigenfunction expansions associated with second-order
  differential equations}.
\newblock Claredon Press, 1962.

\bibitem{Tr24}
Emmanuel Tr{\'e}lat.
\newblock Control in finite and infinite dimension.
\newblock {\em arXiv:2312.15925}, 2023.

\bibitem{We56}
Alexander~D Wentzell.
\newblock Semigroups of operators corresponding to a generalized differential
  operator of second order.
\newblock {\em Doklady Academii Nauk SSSR (In Russ.)}, 111(2):269--272, 1956.

\bibitem{We59}
Alexander~D Wentzell.
\newblock On boundary conditions for multidimensional diffusion processes.
\newblock {\em Theory of Probability \& Its Applications}, 4(2):164--177, 1959.

\end{thebibliography}

\end{document}